\documentclass[11pt]{amsart}
\usepackage{amsopn}
\usepackage{amssymb, amscd}
\usepackage{multirow}
\usepackage{verbatim}
\usepackage{easybmat}
\usepackage{graphicx}
\usepackage{arydshln}
\usepackage{amsfonts}
\usepackage[latin1]{inputenc}

\usepackage{hyperref}
\usepackage{color}

\newcommand{\nc}{\newcommand}

\nc{\fg}{\mathfrak{f} } \nc{\vg}{\mathfrak{v} } \nc{\wg}{\mathfrak{w} }
\nc{\zg}{\mathfrak{z} } \nc{\ngo}{\mathfrak{n} } \nc{\kg}{\mathfrak{k} }
\nc{\mg}{\mathfrak{m} } \nc{\lgo}{\mathfrak{l} } \nc{\ggo}{\mathfrak{g} }
\nc{\ggob}{\overline{\mathfrak{g}} } \nc{\sog}{\mathfrak{so} }
\nc{\sug}{\mathfrak{su} } \nc{\spg}{\mathfrak{sp} } \nc{\slg}{\mathfrak{sl} }
\nc{\glg}{\mathfrak{gl} } \nc{\cg}{\mathfrak{c} } \nc{\rg}{\mathfrak{r} }
\nc{\hg}{\mathfrak{h} } \nc{\tg}{\mathfrak{t} } \nc{\ug}{\mathfrak{u} }
\nc{\dg}{\mathfrak{d} } \nc{\ag}{\mathfrak{a} } \nc{\pg}{\mathfrak{p} }
\nc{\sg}{\mathfrak{s} } \nc{\bg}{\mathfrak{b} }

\nc{\pca}{\mathcal{P}} \nc{\nca}{\mathcal{N}} \nc{\lca}{\mathcal{L}}
\nc{\oca}{\mathcal{O}} \nc{\mca}{\mathcal{M}} \nc{\tca}{\mathcal{T}}
\nc{\aca}{\mathcal{A}} \nc{\cca}{\mathcal{C}} \nc{\gca}{\mathcal{G}}
\nc{\sca}{\mathcal{S}} \nc{\hca}{\mathcal{H}} \nc{\bca}{\mathcal{B}}
\nc{\dca}{\mathcal{D}} \nc{\val}{\operatorname{val}}

\nc{\vp}{\varphi} \nc{\ddt}{\frac{d}{dt}} \nc{\dds}{\frac{d}{ds}}
\nc{\dpar}{\frac{\partial}{\partial t}} \nc{\im}{\mathtt{i}}

\nc{\SO}{\mathrm{SO}} \nc{\Sp}{\mathrm{Sp}} \nc{\Sl}{\mathrm{SL}}
\nc{\SU}{\mathrm{SU}} \nc{\Or}{\mathrm{O}} \nc{\U}{\mathrm{U}} \nc{\Gl}{\mathrm{GL}}
\nc{\Se}{\mathrm{S}} \nc{\Cl}{\mathrm{Cl}} \nc{\Spein}{\mathrm{Spin}}
\nc{\Pin}{\mathrm{Pin}} \nc{\G}{\mathrm{GL}_n(\RR)} \nc{\g}{\mathfrak{gl}_n(\RR)}
\nc{\Span}{\mathrm{Span}}

\nc{\RR}{\mathbb{R}}
\nc{\CC}{ \mathbb{C}}
\nc{\ZZ}{{\mathbb Z}}
\nc{\NN}{ \mathbb{N}}

\nc{\vs}{\vspace{.2cm}} \nc{\vsp}{\vspace{1cm}} \nc{\ip}{\langle\cdot,\cdot\rangle}
\nc{\unc}{\tfrac{1}{4}} \nc{\und}{\tfrac{1}{16}} \nc{\no}{\vs\noindent}
\nc{\ben}{\begin{enumerate}} \nc{\een}{\end{enumerate}} \nc{\f}{\frac}
\nc{\wt}{\widetilde} \nc{\mm}{M}

\nc{\Hm}{{H}}

\nc{\ad}{\operatorname{ad}}
\nc{\Ad}{\operatorname{Ad}}
 \nc{\End}{\operatorname{End}}
\nc{\Id}{\operatorname{Id}}
\nc{\Der}{\operatorname{Der}} \nc{\Ker}{\operatorname{Ker}}
\nc{\Iso}{\operatorname{I}} \nc{\Diff}{\operatorname{Diff}}
\nc{\tr}{\operatorname{tr}}
\nc{\sen}{\operatorname{sen}}
 \nc{\Cric}{\operatorname{P}} \nc{\Ricci}{\operatorname{Ric}}
\nc{\Sym}{\operatorname{Sym}}
 \nc{\ricci}{\operatorname{Rc}}
 \nc{\Rin}{\operatorname{M}}
 \nc{\Diag}{\operatorname{Diag}}
\nc{\Spec}{\operatorname{Spec}}
\newcommand{\mi}{\mathbf{i}}

\newcommand{\innerp}[1]{\langle\!\langle #1 \rangle\!\rangle}
\newcommand{\op}{\operatorname}

\numberwithin{equation}{section}

\theoremstyle{plain}
\newtheorem{theorem}{Theorem}[section]
\newtheorem{proposition}[theorem]{Proposition}

\newtheorem{lemma}[theorem]{Lemma}

\theoremstyle{definition}

\theoremstyle{remark}
\newtheorem{remark}[theorem]{Remark}

\title[Negative Ricci curvature]{Non-solvable Lie groups with negative Ricci curvature}
\author{Emilio~A.~Lauret}
\address{Instituto de Matem\'atica (INMABB), Departamento de Matem\'atica, Universidad Nacional del Sur (UNS)-CONICET, Bah\'ia Blanca, Argentina.}
\email{emilio.lauret@uns.edu.ar}

\author{Cynthia E.~Will}
\address{Universidad Nacional de C\'ordoba, FaMAF and CIEM, 5000 C\'ordoba, Argentina.}
\email{cwill@famaf.unc.edu.ar}

\thanks{This research was partially supported by Grants from CONICET, FONCYT and SeCyT (Universidad Nacional de Córdoba).
The first named author was also supported by the Alexander von Humboldt Foundation (return fellowship).}
\date{May 29, 2019}

\begin{document}

\begin{abstract}
Until a couple of years ago, the only known examples of Lie groups admitting left-invariant metrics with negative Ricci curvature were either solvable or semisimple.
We use a general construction from a previous article of the second named author to produce a large amount of examples with compact Levy factor.
Given a compact semisimple real Lie algebra $\mathfrak u$ and a real representation $\pi$ satisfying some technical properties, the construction returns
a metric Lie algebra $\mathfrak l(\mathfrak u,\pi)$ with negative Ricci operator.
In this paper, when $\mathfrak u$ is assumed to be simple, we prove that $\mathfrak l(\mathfrak u,\pi)$ admits a metric having negative Ricci curvature for all but finitely many finite-dimensional irreducible representations of $\mathfrak u\otimes_{\mathbb R}\mathbb C$, regarded as a real representation of $\mathfrak u$.
We also prove in the last section a more general result where the nilradical is not abelian, as it is in every $\mathfrak l(\mathfrak u,\pi)$.
\end{abstract}

\maketitle

\section{Introduction}\label{intro}

A natural question that has inspired a great deal of research is what can be said  about a differentiable manifold $M$ that admits a metric with curvature of a certain sign.
In the homogeneous case, while the scalar and sectional curvature behavior are settled, the question of when a homogeneous manifold admits a left-invariant metric with
negative Ricci curvature
seems far from being understood (see e.g.\ the introductions of \cite{LD} or \cite{NN}).

In this paper we are interested in the case when $M$ is a Lie group. We note that in this case one can study the problem at the Lie algebra level.
There are three kinds of examples in the literature of Lie groups admitting  metrics with negative Ricci curvature.

Dotti, Leite and Miatello proved in \cite{DLM}
that the only unimodular Lie groups that can admit a left-invariant metric with $\Ricci<0$ are the non-compact semisimple ones.
Furthermore, they where able to show that most of the non-compact simple Lie groups indeed have one with the exception of the following groups:
\begin{equation}\label{excep:groups}
   \begin{array} {c} \Sl(2,\CC),\quad \Sl(3,\CC), \quad \SO(5,\CC),\quad \SO(7,\CC),\quad
\Sp(3,\CC),\quad \Sp(4,\CC),\\ \\ \Sp(5,\CC),\quad
(G_2)_{\CC},\quad \Sl(2,\RR),\quad \Sp(2,\RR),\quad \Sp(3,\RR), \quad G_2
\end{array}
\end{equation}
(cf.\ remark after \cite[Thm.~2.1]{DLM}).
It is known that $\Sl(2,\RR)$ does not admit a Ricci negative metric (see \cite{M}) but the existence of such a metric on the other groups listed above is still open.
It is worth mentioning that Jablonski and Petersen recently proved in \cite{JP}
that such semisimple Lie group cannot have any compact factor.

The second sort of examples are solvable Lie groups.
This is the most developed case (see \cite{D}, \cite{NN}, \cite{N}, \cite{LD}) probably due to the relationship with Einstein solvmanifolds.
Recall that any non-flat Einstein solvmanifold is an example of a Lie group with a metric of negative Ricci curvature.

Let $\sg$ be a solvable Lie algebra admitting an inner product with negative Ricci operator.
Besides the fact that $\sg$ cannot be unimodular (see \cite{D}), Nikolayevsky and Nikonorov~\cite{NN} gave the only necessary condition known so far, namely,
there exists $Y\in\sg$ such that $\tr{\ad{Y}}>0$ and every eigenvalue of $\ad{Y}_{|_{\zg(\ngo)}}$ have positive real part, where $\ngo$ is the nilradical of $\sg$ and $\zg(\ngo)$ is the center of $\ngo$.

The second named author constructed in \cite{u2} the first examples of Lie groups, which are not solvable nor semisimple, admitting a left-invariant metric with negative Ricci curvature.
This construction was extended to a more general setting in \cite{CRN} that we next describe.

Let $\ug$ be a compact semisimple Lie algebra.
Given $(\pi, V)$ a  finite-dimensional real representation of $\mathfrak u$, let $\mathfrak l(\mathfrak u,\pi)$ be the Lie algebra defined by
\begin{equation}\label{l}
\mathfrak l(\mathfrak u,\pi):= (\RR Z \oplus \ug) \ltimes V
\quad\text{determined by } \quad
{\ad Z}_{|_\ug}=0 \text{ and }
{\ad Z}_{|_V}=\Id.
\end{equation}
It was proven in \cite[Thm.~3.3]{CRN} that, if $(\pi,V)$ satisfies some technical conditions, the Lie algebra $\mathfrak l(\mathfrak u,\pi)$ admits an inner product with negative definite Ricci operator.
See Theorem~\ref{theTheo} for the precise statement.

Furthermore, the article \cite{CRN} provides explicit examples of the above construction by using representations of the classical compact real Lie algebras ($\mathfrak{su}(n)$,  $\mathfrak{so}(n)$ and $\mathfrak{sp}(n)$ for $n\geq2$) realized as a vector space of homogeneous polynomials.

The main goal of this article is to find systematic ways to construct metric Lie algebras with negative Ricci operator from Theorem~\ref{theTheo}.
From now on we will consider complex representations of the complexified Lie algebra $\mathfrak u_\CC:=\mathfrak u\otimes_\RR\CC$, viewed as real representations of $\mathfrak u$ by restriction and by restricting scalars.
More precisely, a  finite-dimensional complex representation $\pi:\mathfrak u_\CC\to \mathfrak{gl}(V)$ induces the real representation $\pi_{|_{\mathfrak u}}:\mathfrak u\to \mathfrak{gl}(V)$, where $V$ is regarded as a real vector space. This representation will still be denoted by $(\pi,V)$.

We develop three essentially different approaches to decide whether such a representation satisfies the hypotheses in Theorem~\ref{theTheo}.
The corresponding representations are given in Theorems~\ref{thm:Weylchamber}, \ref{thm:S=pesosextremales}, and \ref{thm:zeroweight}.
The first approach is called \emph{the Weyl chamber approach} because it assumes the existence of a weight of the representation inside a Weyl chamber.
This allows us to prove the  main result of the article.

\begin{theorem}\label{thm:main-simple}
Let $\mathfrak u $ be a compact real Lie algebra such that its complexified Lie algebra $\mathfrak u_\CC$ is simple.
For all but finitely many finite-dimensional irreducible complex representations $(\pi,V)$ of $\mathfrak u_\CC$, the real Lie algebra $\mathfrak l(\mathfrak u,\pi)$ given by \eqref{l}  admits a metric with negative Ricci curvature.
\end{theorem}

 Furthermore, for $\mathfrak u$ an arbitrary compact real semisimple Lie algebra, Theorem~\ref{thm:main-semisimple} ensures that $\mathfrak l(\mathfrak u,\pi)$ admits an inner product with negative Ricci curvature for infinitely many finite-dimensional irreducible complex representations $(\pi,V)$ of $\mathfrak u_\CC$.

Although the Weyl chamber approach  (Theorem~\ref{thm:Weylchamber}) provides a great amount of examples, it does not work very well for \emph{small} representations.
We then show two more constructions, \emph{the Weyl group orbit approach} and \emph{the zero weight approach}.
The first one was inspired by the constructions in \cite{CRN}, where it was used implicitly.

In Section~\ref{sec:examples} we show many explicit examples.
In particular, when $\mathfrak u_\CC$ has rank at most two, we classify all irreducible representations $\pi$ of $\mathfrak u_\CC$ where Theorem \ref{theTheo} can be applied.
We obtain that the above holds for every such $\pi$ with the only exceptions of the $7$-dimensional irreducible representation of the complex simple Lie algebra of type $\textup{G}_2$, and the fundamental representations when the type is $\textup{A}_2$ or $\textup{B}_2=\textup{C}_2$.

With all these examples in mind, it is clear that the set of Lie groups admitting left-invariant metrics of negative Ricci curvature is much bigger than it seemed and the classification problem  has become more complicated.

We show in the last section that one can get examples with non-abelian nilradical. To be able to do that we need to show a slightly more general version of \cite[Thm.~3.3]{CRN} (see Theorem~\ref{Z:positive}).
More precisely, we consider a Lie algebra $(\RR Z \oplus \mathfrak{u}) \ltimes \mathfrak{n}$,  where $\mathfrak{u}$  is a compact semisimple Lie algebra acting on $\mathfrak{n}$ by derivations and $[Z, \mathfrak{u}]=0$.
Note that in order to get such a Lie algebra, $\ad Z$ must be a derivation of $\ngo$ and therefore it can never act as a multiple of the identity unless $\ngo$ is abelian.

It is known that there is no topological obstruction on a differential manifold to the existence of a complete Riemannian metric with negative Ricci curvature (see \cite{Lo}). Nevertheless, the situation changes when dealing with left-invariant metrics on Lie groups.
First, recall that if $K$ is a maximal compact subgroup of a Lie group $G$, then all the nontrivial topology of $G$ is in $K$, in the sense that as a differentiable manifold, $G$ is the product $K \times \RR^n$. Therefore, from the semisimple examples in \cite{DLM}, it follows that almost all compact simple Lie groups may appear as maximal compact subgroups of a Lie group with negative Ricci curvature with the exceptions (\ref{excep:groups}) namely,
\begin{equation}
\SO(2), \quad \SU(2), \quad \SU(3),\quad \Sp(2),\quad \Sp(3), \quad \Sp(4),\quad \Sp(5),\quad \SO(7), \quad G_2^c.
\end{equation}
Here, $G_2^c$ denotes the simply connected compact Lie group with the Lie algebra of type $\textup{G}_2$.
In \cite{u2} and \cite{CRN} were covered all the groups in the above list with the only exception of $G_2^c$, which is now attained by Theorem~\ref{thm:main-simple} (see also Proposition~\ref{prop:rank2}).
Therefore, as in the general case, there are (almost) no topological obstructions for the existence of a left-invariant metric of negative Ricci curvature on a Lie group.

The paper is organized as follows.
In Section~\ref{sec:preliminaries} we recall all the facts about compact Lie groups and their representations that will be used in the sequel.
Section~\ref{sec:approaches} contains the three approaches to ensure that $\mathfrak l(\mathfrak u,\pi)$ admits a metric of negative Ricci curvature, as well
as the proof of the main theorem.
Several explicit examples are shown in Section~\ref{sec:examples}.
In the last section, we construct Lie algebras with the non-abelian nilradical admitting an inner product of negative Ricci curvature.

\subsection*{Acknowledgments}
The authors wish to thank Jorge R.\ Lauret  and the referees for several helpful comments on a preliminary version of this paper.

\section{Preliminaries}\label{sec:preliminaries}

\textit{Throughout the paper, all Lie algebras as well as their representations are assumed finite dimensional.}

In this section we first recall well-known facts on compact real forms and representations of a complex semisimple Lie algebra via root systems.
Very good general references are \cite[\S{}2.4 and \S{}3.1--2]{GoodmanWallach-book-Springer} and \cite[Ch.\ II and V]{Kn}.

\subsection{Root system}\label{subsec:rootsystem}

Let $\mathfrak g$ be a complex semisimple Lie algebra and $B$ its Killing form.
We fix a Cartan subalgebra $\mathfrak h$ of $\mathfrak g$, and let $\Delta$ denote the corresponding system of roots and $W$ the Weyl group.
One has the root decomposition
\begin{equation}
\mathfrak g=\mathfrak h\oplus\bigoplus_{\alpha\in\Delta} \mathfrak g_\alpha,
\end{equation}
where $\mathfrak g_\alpha=\{X\in\mathfrak g: [H,X]=\alpha(H)\text{ for all }H\in\mathfrak h\}$ is one dimensional for all $\alpha\in\Delta$.

For $\alpha\in \Delta$, we denote by $H_\alpha$ the corresponding coroot, that is, the only element in $\mathfrak h$ such that $B(H,H_\alpha) = \alpha(H)$ for all $H\in\mathfrak h$.
We denote by $\mathfrak h_\RR$ the $\RR$-linear span of all $H_\alpha$ for $\alpha\in\Delta$, which is a real form of $\mathfrak h$.
We will consider the inner product on $\mathfrak h_\RR^*$ determined by $\langle \alpha,\beta\rangle :=B(H_\alpha,H_\beta)$.

For each $\alpha \in \Delta$, it is possible to choose $X_\alpha \in \mathfrak g_\alpha$ such that, for all $\alpha, \beta \in \Delta$,
\begin{equation}\label{sperootvect}
\begin{aligned}{}
[X_{\alpha} , X_{-\alpha}]
&= H_\alpha, \\
[X_{\alpha},  X_{\beta}] &= N_{\alpha,\beta} X_{\alpha+\beta}&&\text{if } \alpha+\beta \in \Delta,
\\
[X_{\alpha},  X_{\beta}] &=0 &&\text{if } \alpha+\beta \ne 0\text{ and }  \alpha+\beta \notin \Delta,
\end{aligned}
\end{equation}
where $N_{\alpha,\beta} = - N_{-\alpha,-\beta} \in \RR$ (see \cite[Thm~6.6]{Kn}).
It turns out that
\begin{equation}\label{genu}
\ug := \displaystyle{\sum_{\alpha \in \Delta}} \RR \mi H_\alpha + \sum_{\alpha \in \Delta} \RR (X_{\alpha} - X_{-\alpha}) + \sum_{\alpha \in \Delta} \RR \mi (X_{\alpha} + X_{-\alpha}) \end{equation}
is a \emph{compact real form of $\mathfrak g$}, that is, a real Lie algebra whose complexified Lie algebra is $\mathfrak g$, which has negative definite Killing form (see \cite[Thm.~6.11]{Kn}).

We denote $H^\alpha=\mi H_\alpha$, $X^\alpha = (X_{\alpha} - X_{-\alpha})$ and $Y^\alpha = \mi (X_{\alpha} + X_{-\alpha})$ for each $\alpha\in\Delta$.
For $\alpha,\beta\in\Delta$, one has that
\begin{equation}\label{lbroot}
\begin{aligned}{}
[H^\alpha, X^\beta] &= c_{\alpha,\beta}  Y^\beta, &
[X^\alpha, X^\beta] &= N_{\alpha,\beta} X^{\alpha+\beta}-  N_{-\alpha,\beta} X^{-\alpha+\beta}
	&&\text{if } \beta \ne \pm \alpha,\\
[H^\alpha, Y^\beta] &= - c_{\alpha,\beta} X^\beta,&
[Y^\alpha, Y^\beta] &= - N_{\alpha,\beta} X^{\alpha+\beta} -  N_{-\alpha,\beta} X^{-\alpha+\beta}
	&&\text{if } \beta \ne \pm \alpha,\\
[X^\alpha, Y^\alpha] &= 2 H^\alpha, &
[X^\alpha,  Y^\beta] &= N_{\alpha,\beta} Y^{\alpha+\beta}-  N_{-\alpha,\beta} Y^{-\alpha+\beta}
	&&\text{if } \beta \ne \pm \alpha,
\end{aligned}
\end{equation}
where $c_{\alpha,\beta}$ are real numbers and $N_{\alpha,\beta}=0$ if $\alpha+\beta \notin \Delta$.

\subsection{Weights of representations} \label{subsec:weights}
Let $(\pi,V)$ be a complex representation of $\mathfrak g$, that is, a $\CC$-linear map $\pi:\mathfrak g\to \mathfrak{gl}(V)$ satisfying that $\pi([X,Y])=\pi(X)\pi(Y)-\pi(Y)\pi(X)$ for all $X,Y\in\mathfrak g$.
We recall our convention that all representations are assumed finite-dimensional.

One has the weight decomposition
\begin{equation}\label{eq:weightdecomposition}
V=\bigoplus_\mu V(\mu),
\end{equation}
where $\mu\in \mathfrak h^*$ and
\begin{equation}\label{eq:V(mu)}
V(\mu)=\{v\in V: \pi(H)v=\mu(H)v\text{ for all }H\in\mathfrak h\}.
\end{equation}
Those $\mu\in \mathfrak h^*$ satisfying $V(\mu)\neq0$ are called the \emph{weights} of $\pi$ and belong to the weight lattice $P(\mathfrak g):=\{\mu\in\mathfrak h^*: \langle \mu,\alpha\rangle/\langle\alpha,\alpha\rangle \in\ZZ\text{ for all }\alpha\in\Delta\}\subset \mathfrak h_\RR^*$.
Furthermore,
\begin{equation}\label{eq:E_alphaV(mu)}
\pi(\mathfrak g_\alpha)V(\mu)\subset V(\mu+\alpha)
\end{equation}
for $\alpha\in\Delta$ and $\mu\in P(\mathfrak g)$.
The dimension of $V(\mu)$ is called the \emph{weight multiplicity} of $\mu$ in $\pi$.
One has $\dim V(w\cdot \mu)=\dim V(\mu)$ for all $w\in W$.

We now pick a positive system $\Delta^+$ of $\Delta$.
Let $\Pi$ be the set of simple roots.
An element $\mu\in\mathfrak h_\RR^*$ is called \emph{dominant} if $\langle \mu,\alpha\rangle \geq0$ for all $\alpha\in\Delta^+$.
Each connected component of the complement of the set of hyperplanes $\{\mu\in\mathfrak h_\RR^*: \langle \mu,\alpha\rangle=0\}$ for $\alpha\in\Delta$ is called a \emph{Weyl chamber}.
The cone of dominant elements coincides with the closure of the \emph{fundamental Weyl chamber} $\{\mu\in\mathfrak h_\RR^*: \langle \mu,\alpha\rangle>0\text{ for all } \alpha\in\Delta^+\}$.

Let $P^+(\mathfrak g)$ denote the set of dominant elements in the weight lattice $P(\mathfrak g)$.
We write $\Pi=\{\alpha_1,\dots,\alpha_n\}$ ($n$ is the rank of $\mathfrak g$).
The \emph{fundamental weights} $\omega_1,\dots,\omega_n$ are given by $2\langle \omega_i,\alpha_j\rangle/\langle \alpha_j,\alpha_j\rangle =\delta_{i,j}$ and satisfy $P(\mathfrak g)=\bigoplus_{i=1}^n \ZZ \omega_i$ and $P^+(\mathfrak g)=\bigoplus_{i=1}^n \ZZ_{\geq0} \omega_i$.

The next result will be useful in the sequel (see for instance \cite[Prop.~3.2.11]{GoodmanWallach-book-Springer}).

\begin{lemma}\label{lem:GoodmanWallach}
Let $\mu\in P(\mathfrak g)$ dominant  and $\nu$ a weight of $V$.
If $\nu-\mu$ can be written as a sum of positive roots, then $\mu$ is a weight of $V$.
\end{lemma}

The Highest Weight Theorem parameterizes irreducible complex representations of $\mathfrak g$ with dominant elements in $P(\mathfrak g)$.
For $\lambda\in P(\mathfrak g)$ dominant, we write $(\pi_\lambda,V_\lambda)$ the corresponding irreducible complex representation of $\mathfrak g$ with highest weight $\lambda$.

We conclude the preliminaries section by showing the existence of a weight of a representation in a Weyl chamber, for all but finitely many irreducible representations of a complex simple Lie algebra.
Although this result may be present in the mathematical literature, we include a case-by-case proof because the authors could not find any reference.
The reader will find two shorter and uniform proofs in the answers by Sam Hopkins and David E~Speyer of the MathOverflow question \cite{Mathoverflow}.

\begin{lemma}\label{lem:allbutfinitelymany}
Let $\mathfrak g$ be a complex simple Lie algebra.
For all but finitely many complex irreducible representations $(\pi,V)$ of $\mathfrak g$, we have that $(\pi,V)$ contains a weight in a Weyl chamber.
\end{lemma}

\begin{proof}
By the Highest Weight Theorem, every irreducible representation of $\mathfrak g$ is in correspondence with an element in $\{\sum_{j=1}^n a_j\omega_j: a_j\in \ZZ_{\geq0}\text{ for all $j$}\}$.
Since $\mathcal P_r:=\{\sum_{j=1}^n a_j\omega_j: a_j\in\ZZ \text{ and } 0\leq a_j\leq r\text{ for all $j$}\}$ is finite for every $r\geq0$, it is sufficient to prove that, for $r$ sufficiently large, every irreducible representation of $\mathfrak g$ with highest weight $\lambda \in P^+(\mathfrak g)\smallsetminus \mathcal P_r$ has a weight in the fundamental Weyl chamber.
	
	To do that, we will use Lemma~\ref{lem:GoodmanWallach}.
	More precisely, for such $\lambda$, we will show that
	\begin{quotation}
		\it
		there are $\beta_1,\dots,\beta_l\in\Delta^+$ such that $\lambda-(\beta_1+\dots+\beta_l)$ is in the fundamental Weyl chamber.
	\end{quotation}
	
	Clearly, it is sufficient to show that there is an integer $r$ big enough satisfying that
	\begin{quotation}
		\it
		for each $1\leq i\leq n$, there are  $\beta_1,\dots,\beta_l\in\Delta^+$ such that $r\omega_i-(\beta_1+\dots+\beta_l)$ is dominant and its $\omega_j$-coefficient is positive for all $1\leq j\leq n$.
	\end{quotation}

The checking process of the above condition involves only computations in the corresponding irreducible root system.
We will do it case by case.
	
We start considering complex simple Lie algebras of exceptional type.
We will use the data available in \cite[\S{}C.2]{Kn}, where the fundamental weights $\omega_1,\dots,\omega_n$ are written in terms of the simple roots $\alpha_1,\dots,\alpha_n$.

\begin{description}
\item[Type G$_2$]
In this case, $n=2$,
$\omega_1=2\alpha_1+\alpha_2$, and $\omega_2=3\alpha_1+2\alpha_2$.
It follows immediately that
$r\omega_1-\alpha_1-\alpha_2 = (r-2)\omega_1 + \omega_2$ and $r\omega_2-\alpha_1-\alpha_2 = (r-1)\omega_2 + \omega_1$, thus the required condition holds with $r=3$.
		
\item[Type F$_4$]
In this case, $n=4$, and the fundamental weights are
$\omega_1=2\alpha_1+3\alpha_2+2\alpha_3+ \alpha_4$,
$\omega_2=3\alpha_1+6\alpha_2+4\alpha_3+2\alpha_4$,
$\omega_3=4\alpha_1+8\alpha_2+6\alpha_3+3\alpha_4$, and
$\omega_4=2\alpha_1+4\alpha_2+3\alpha_3+2\alpha_4$.
It follows that
\begin{align*}
r\omega_1-5\alpha_1-2\alpha_2-\alpha_3
&= (r-7)\omega_1+(2\omega_1-\alpha_1)+ (3\omega_1-2\alpha_1)
\\ & \quad
+(2\omega_1-2\alpha_1-2\alpha_2-\alpha_3) \\
&= (r-7)\omega_1+\omega_2+\omega_3+\omega_4,
\end{align*}%
\begin{align*}
r\omega_2-4\alpha_1-9\alpha_2-5\alpha_3-2\alpha_4&= (r-4)\omega_2+\omega_1+\omega_3+\omega_4,
\\
r\omega_3-3\alpha_1-8\alpha_2-7\alpha_3-3\alpha_4
&= (r-3)\omega_3+\omega_1+\omega_2+\omega_4,
\\
r\omega_4-\alpha_1-3\alpha_2-3\alpha_3-4\alpha_4
&= (r-5)\omega_4+\omega_1+\omega_2+\omega_3.
\end{align*}
Hence, the required condition holds with $r=8$.

\item[Type E$_n$]
In these cases, $n=6,7,8$ and, for each $1\leq i\leq n$, $\omega_i= \sum_{j=1}^n \frac{a_j}{2} \alpha_j$ with $a_1,\dots,a_n$ positive integers.
Although it is pretty involved to give the explicit calculations as in the previous two cases, it is clear that for a positive integer $r$ sufficiently large, one has
\begin{equation*}
r\omega_i= (r-s)\omega_i +\sum_{j=1}^n b_j\omega_j + \mu
\end{equation*}
for some integer $s\leq r$, $b_1,\dots,b_n$ positive integers, and $\mu$ a sum of simple roots.
Hence, the condition is valid.
\end{description}

We next consider the classical Lie algebras.
The root system data can be found in \cite[\S{}C.1]{Kn}.
We give all the details for type $\textup{D}_n$ since it is the one that presents more difficulties.
The rest of the types are given in a brief way because the method is analogous.
	
\begin{description}
\item[Type D$_n$]
In this case, $\mathfrak g=\mathfrak{so}(2n,\CC)$ for any $n\geq 3$, $\mathfrak h_\RR= \Span_\RR\{\varepsilon_1,\dots,\varepsilon_n\}$,
$
\Delta^+=\{\varepsilon_i\pm\varepsilon_j : 1\leq i<j\leq n\},
$
$\omega_i = \varepsilon_1+\dots+\varepsilon_i$ for $1\leq i\leq n-2$,
$\omega_{n-1} = \tfrac12 (\varepsilon_1+ \dots+ \varepsilon_{n-1} -\varepsilon_n)$, and
$\omega_{n} = \tfrac12 (\varepsilon_1+\dots+\varepsilon_n)$.
We have that
\begin{align}
\label{Di=1}
r\omega_1-(\varepsilon_1-\varepsilon_2)
	&=(r-1)\omega_1 +\varepsilon_2 = (r-2)\omega_1+\omega_2,
\\ \label{Di=2}
r\omega_i-(\varepsilon_i-\varepsilon_{i+1})
	&= (r-1)\omega_i + (\varepsilon_1+\dots+\varepsilon_i) -(\varepsilon_i-\varepsilon_{i+1})
\\ \notag
	&= \omega_{i-1} + (r-2)\omega_i +\omega_{i+1}
	\qquad \qquad \text{for $2\leq i\leq n-3$},
\\ \label{Di=n-2}
r\omega_{n-2}-(\varepsilon_{n-2}-\varepsilon_{{n-1}})
	&= (r-1)\omega_{n-2} +(\varepsilon_1+\dots+\varepsilon_{n-3})+\varepsilon_{n-1}
\\ \notag
	&= \omega_{n-3} + (r-2)\omega_{n-2} + (\varepsilon_1+\dots+\varepsilon_{n-1})
\\ \notag
	&= \omega_{n-3} + (r-2)\omega_{n-2} +\omega_{n-1}+\omega_n,
\end{align}
Furthermore,
\begin{align}
r\omega_{n-1}-3(\varepsilon_{n-1}-\varepsilon_{n}) -(\varepsilon_{n-2}-\varepsilon_{n-1})
	&= (r-6) \omega_{n-1} +3\omega_{n-2}-(\varepsilon_{n-2}-\varepsilon_{n-1})
\\ \notag
	&=  \omega_{n-3}+\omega_{n-2}+(r-5)\omega_{n-1}+\omega_n
\\
r\omega_n-2(\varepsilon_{n-1}+\varepsilon_n) - (\varepsilon_{n-2} - \varepsilon_{n-1})
	&= (r-4)\omega_n + 2\omega_{n-2} - (\varepsilon_{n-2} - \varepsilon_{n-1})
\\ \notag
	&=   \omega_{n-3}+\omega_{n-1} +(r-3)\omega_n.
\end{align}

		The above identities tell us that, for any $1\leq i\leq n$, one can subtract to $r\omega_i$ a sum of positive roots obtaining a dominant element with positive $\omega_{i-1}$ and $\omega_{i+1}$-coefficients (it is understood that there are no $\omega_{i}$-coefficient for $i=0,n+1$).
		Proceeding in this way several times, one obtains a dominant element in $P(\mathfrak g)$ with positive $\omega_j$-coefficient for all $j$, for $r$ large enough.
		This proves the required condition.

		\item[Type B$_n$]
		$\mathfrak g=\mathfrak{so}(2n+1,\CC)$ for any $n\geq 2$, $\mathfrak h_\RR= \Span_\RR\{\varepsilon_1,\dots,\varepsilon_n\}$, $\Delta^+=\{\varepsilon_i\pm\varepsilon_j : 1\leq i<j\leq n\} \cup\{\varepsilon_i: 1\leq i\leq n\}$, $\omega_i =\varepsilon_1+\dots+\varepsilon_i$ for $1\leq i\leq n-1$, and $\omega_n=\tfrac12 (\varepsilon_1+\dots+\varepsilon_n)$.
		
		One can check that \eqref{Di=1} holds, as well as \eqref{Di=2} for $2\leq i\leq n-2$.
		Furthermore,
		$r\omega_{n-1}-(\varepsilon_{n-1}-\varepsilon_n) = \omega_{n-2}+(r-2)\omega_{n-1}+2\omega_n$, and
		$r\omega_n- \varepsilon_n= \omega_{n-1} + (r-2)\omega_n$.

		\item[Type C$_n$]
		$\mathfrak g=\mathfrak{sp}(n,\CC)$ for any $n\geq 2$, $\mathfrak h_\RR= \Span_\RR\{\varepsilon_1,\dots,\varepsilon_n\}$, $\Delta^+=\{\varepsilon_i\pm\varepsilon_j : 1\leq i<j\leq n\} \cup\{ 2\varepsilon_i: 1\leq i\leq n\}$, $\omega_i =\varepsilon_1+\dots+\varepsilon_i$ for $1\leq i\leq n$.
		
		One can check that \eqref{Di=1} holds, as well as \eqref{Di=2} for $2\leq i\leq n-1$.
		Furthermore, $r\omega_n- 2\varepsilon_n= 2\omega_{n-1} + (r-2)\omega_n$.

		\item[Type A$_n$]
		$\mathfrak g=\mathfrak{sl}(n+1,\CC)$ for any $n\geq 1$, $\mathfrak h_\RR=\{\sum_{i=1}^{n+1} a_i\varepsilon_i: a_i\in\RR \;\forall \, i
		,\; \sum_{i=1}^{n+1}a_i=0 \}$, $\Delta^+=\{\varepsilon_i-\varepsilon_j : 1\leq i<j\leq n+1\}$, $\omega_i =\varepsilon_1+\dots+\varepsilon_i-\tfrac{i}{n+1} (\varepsilon_1+\dots+\varepsilon_{n+1})$ for $1\leq i\leq n$.
		
		One can check, for any $1\leq i\leq n$, that
		\begin{align*}
			r\omega_i-(\varepsilon_i-\varepsilon_{i+1})
			&= (r-2)\omega_{i-1}+2(\varepsilon_1+\dots+\varepsilon_i) - \tfrac{2i}{n+1} (\varepsilon_1+\dots+\varepsilon_{n+1}) -\varepsilon_i+\varepsilon_{i+1} \\
			&= (r-2)\omega_{i-1}+(\varepsilon_1+\dots+\varepsilon_{i-1}) -\tfrac{i-1}{n+1} (\varepsilon_1+\dots+\varepsilon_{n+1}) \\
			&\quad + (\varepsilon_1+\dots+\varepsilon_{i+1}) - \tfrac{i+1}{n+1} (\varepsilon_1+\dots+\varepsilon_{n+1}) \\
			&= \omega_{i-1} + (r-2)\omega_i +\omega_{i+1} ,
		\end{align*}
		where $\omega_{0}=\omega_{n+1}=0$.
	\end{description}
	
	We conclude that the required condition holds for every complex simple Lie algebra, which completes the proof.
\end{proof}


\section{Existence of Ricci negative metric Lie algebras} \label{sec:approaches}

In this section we introduce three different approaches to use Theorem~\ref{theTheo}.
Furthermore, it also contains the proof of the main theorem.

In what follows, $\mathfrak u$ denotes a compact semisimple real Lie algebra and, since a compact real form of a semisimple complex Lie algebra is unique up to isomorphism (see \cite[Cor.~6.20]{Kn}), there is a root system $\Delta$ of the complex Lie algebra $\mathfrak u_\CC:= \mathfrak u\otimes_\RR\CC$ and elements $X_\alpha \in (\mathfrak u_\CC)_\alpha$ for each $\alpha\in\Delta$ such that the compact real form given by \eqref{genu} coincides with $\mathfrak u$.
Furthermore, we fix
$\Delta^+$ a positive system for $\Delta$ and $\Pi= \{\alpha_1,\dots,\alpha_n\}$ an ordered set of simple roots.

In the sequel, we will use the objects introduced in Section~\ref{sec:preliminaries} without further comments, for instance, the weight lattice $P(\mathfrak u_\CC)$, its dominant elements $P^+(\mathfrak u_\CC)$, weights of a complex representation, the Weyl group $W$, Weyl chambers, the fundamental weights $\omega_1,\dots,\omega_n$, the (unique up to equivalence) irreducible representation $(\pi_\lambda,V_\lambda)$ with the highest weight $\lambda\in P^+(\mathfrak u_\CC)$, etc.

\subsection{Ricci negative inner products}

We first recall the main theorem in \cite{CRN}, which will be the main tool in the sequel.

To a given real representation $(\pi,V)$ of a real Lie algebra $\mathfrak u$ we associate the real Lie algebra $\mathfrak l(\mathfrak u,\pi) := (\RR Z\oplus \mathfrak u)\ltimes V$, where $\RR Z\oplus\mathfrak u$ is a central extension of $\mathfrak u$ (i.e.\ ${\ad Z}_{|_\ug}=0$) and $Z$ acts as the identity on $V$ (i.e.\ ${\ad Z}_{|_V}=\Id$).
More precisely, the brackets in $\mathfrak l(\mathfrak u,\pi)$ are determined by
\begin{align}\label{l2}
[Z,X]&=0,&
[Z,v]&=v,&
[X,v]&= \pi(X)v,&
[v,w]&=0,
\end{align}
for all $X \in \mathfrak u$ and $v,w\in V$.

\begin{theorem}\label{theTheo} \cite[Thm.~3.3]{CRN}
Let $(V,\pi)$ be a real representation of $\mathfrak u$.
We assume the decomposition $V=V_1 \oplus V_2$ and the existence of an inner product $\langle\cdot,\cdot\rangle$ on $V$ satisfying the following properties:
\begin{enumerate}
\item \label{item:H^alpha-invariant}
$V_1$ and $V_2$ are $\mi \mathfrak h_\RR$-invariant.
		
\item \label{item:X^alphaV_1subsetV_2}
$\pi(X^\alpha)(V_1) \subset V_2$ and $\pi(Y^\alpha)(V_1) \subset V_2$ for every $\alpha \in \Delta^+$.
		
\item \label{item:orthogonal}
$V_1$ is orthogonal to $V_2$.
		
\item \label{item:skewsymmetric}
$\pi(H)$ is a skew-symmetric operator of $V$ for every $H\in \mi\mathfrak h_\RR$.
		
\item \label{item:non-trivial}
$\pi(X^\alpha)_{|_{V_1}}$ and $\pi(Y^\alpha)_{|_{V_1}}$ are not trivial for every $\alpha \in \Delta^+$.
		
\item \label{item:trace=0}
$\tr \pi(Y)_{|_{V_1}}^t \pi(X)_{|_{V_1}} = 0$ whenever $X \ne Y$ are elements of $\{X^\alpha, Y^\alpha,\, \alpha \in \Delta^+\}$.
\end{enumerate}
Then, the real Lie algebra $\mathfrak l(\mathfrak u,\pi)$ given by \eqref{l2} (i.e.\
$\mathfrak l(\mathfrak u,\pi) = (\RR Z\oplus \mathfrak u)\ltimes V$ where $\ad Z_{|_{\mathfrak u}}=0$ and $\ad Z_{|_V}=\op{Id}$) admits an inner product with negative Ricci curvature.
\end{theorem}

The idea of the proof is to first prove that $\lgo (\mathfrak u,\pi)$ degenerates into a solvable Lie algebra $\lgo_\infty$ to finally show that this limit admits an inner product with $\Ricci < 0$. By continuity, so does the starting Lie algebra.

\subsection{Common generalities}
Let $(\pi,V)$ be a complex representation of $\mathfrak u_\CC$.
We can regard $V$ as a real vector space by restricting scalars, and then the restriction of $\pi$ to $\mathfrak u$ becomes a real representation of $\mathfrak u$.
To simplify the notation, the resulting real representation $\pi_{|_{\mathfrak u}}:\mathfrak u\to\mathfrak{gl}(V)$ is again denoted by $(\pi,V)$.

Let $U$ be the simply connected Lie group with Lie algebra $\mathfrak u$.
We denote again by $\pi$ the associated homomorphism of groups $U\to\Gl(V)$.
Since $U$ is compact, there is a complex inner product $\innerp{\cdot,\cdot}$ on the complex vector space $V$ such that $\pi$ is unitary.
More precisely, for $v,w\in V$, one has that $\innerp{\pi(a)v,\pi(a)w}=\innerp{v,w}$ for all $a\in U$, which gives $\innerp{\pi(X)v,w}=-\innerp{v,\pi(X)w}$ for all $X\in \mathfrak u$.
We set
\begin{equation}\label{inner_product}
\langle v,w\rangle = \operatorname{Re}(\innerp{v,w})
\qquad\text{for }v,w\in V.
\end{equation}
This is clearly a real inner product on $V$.
Moreover, $\langle \pi(X)v,w\rangle = -\langle v,\pi(X)w\rangle$ for all $X\in\mathfrak u$ and $v,w\in V$, that is, $\pi(X)$ acts as a skew-symmetric operator on $V$ for all $X\in\mathfrak u$.
From now on, we will associate such real inner product on $V$ to any complex representation of $\mathfrak u_\CC$ without further comments.

It follows immediately from the definition \eqref{eq:V(mu)} that a weight space is invariant by $\mathfrak h_\RR$.
Thus, property \eqref{item:H^alpha-invariant} in Theorem~\ref{theTheo} suggests us to decompose $V=V_1\oplus V_2$, by taking $V_1$ and $V_2$ subspaces given by sums of distinct weight spaces.

The simple observations from the two last paragraphs guarantee us four properties from Theorem~\ref{theTheo}.

\begin{proposition}\label{prop:V_1(S)}
Let $(\pi,V)$ be a complex representation of $\mathfrak u_\CC$ and let $\mathcal S$ be a subset of weights of $\pi$.
We set
\begin{align*}
V_1&=\bigoplus_{\mu\in \mathcal S} V(\mu), &
V_2&=\bigoplus_{\mu\notin \mathcal S} V(\mu).
\end{align*}
Then, there exists a real inner product on $V=V_1\oplus V_2$ such that the induced real representation of $\mathfrak u$ on $V$ satisfies properties \eqref{item:H^alpha-invariant}, \eqref{item:orthogonal}, \eqref{item:skewsymmetric}, and \eqref{item:trace=0} from Theorem~\ref{theTheo}.
Moreover, if in addition $\mathcal S$ has exactly one element, then \eqref{item:X^alphaV_1subsetV_2} also holds.
\end{proposition}

\begin{proof}
Property \eqref{item:H^alpha-invariant} follows from \eqref{eq:V(mu)}.
We already mentioned that $\pi(X):V\to V$ is skew-symmetric on $V$ with respect to $\langle\cdot,\cdot\rangle$ for all $X\in \mathfrak u$, thus property \eqref{item:skewsymmetric} holds.

Let $\mu$ and $\nu$ be distinct weights of $\pi$ and let $v\in V(\mu)$ and $w\in V(\nu)$.
For any $H\in\mathfrak h$, one has that
\begin{equation}
\mu(H) \innerp{v,w} = \innerp{\pi(H)v,w} =-\innerp{v,\pi(H)w}= -\nu(H) \innerp{v,w},
\end{equation}
which implies $\langle v,w\rangle = \op{Re}(\innerp{v,w})=0$ by taking $H\in\mathfrak h$ such that $\mu(H)\neq\nu(H)$.
This shows that different weight spaces are orthogonal with respect to $\langle\cdot,\cdot\rangle$; in particular \eqref{item:orthogonal} holds.

We next prove the validity of property \eqref{item:trace=0}.
Let $\{v_1,\dots, v_d\}$ be an orthogonal $\CC$-basis of $(V,\innerp{\cdot,\cdot})$ given by weight vectors, such that $\{v_1,\dots,v_c\}$ is a $\CC$-basis of $V_1$ ($c\leq d$).
It turns out that $\{v_1,\mi v_1,\dots, v_d,\mi v_d\}$ is an $\RR$-basis of $V$, as well as $\{v_1,\mi v_1,\dots, v_c,\mi v_c\}$ is an $\RR$-basis of $V_1$

We note for $X,Y\in\mathfrak u$ that
\begin{equation}
\begin{aligned}
\tr \pi(Y)_{|_{V_1}}^t \pi(X)_{|_{V_1}}
&= \sum_{j=1}^c \big(\langle \pi(Y)_{|_{V_1}}^t \pi(X)_{|_{V_1}} v_j,v_j\rangle + \langle \pi(Y)_{|_{V_1}}^t \pi(X)_{|_{V_1}} \mi v_j,\mi v_j\rangle\big) \\
&= \sum_{j=1}^c \big(\langle \pi(X)  v_j,\pi(Y) v_j\rangle + \langle  \pi(X)  \mi v_j, \pi(Y)  \mi v_j\rangle\big).
\end{aligned}
\end{equation}
Since $\pi(X^\alpha) V(\mu)\subset V(\mu-\alpha)\oplus V(\mu+\alpha)$, $\pi(Y^\alpha) V(\mu)\subset V(\mu-\alpha)\oplus V(\mu+\alpha)$ by \eqref{eq:E_alphaV(mu)}, and $\langle V(\mu),V(\nu)\rangle=0$ for $\mu\neq \nu$, it follows that $\tr \pi(X^\beta)_{|_{V_1}}^t \pi(X^\alpha)_{|_{V_1}}$, $\tr \pi(Y^\beta)_{|_{V_1}}^t \pi(X^\alpha)_{|_{V_1}}$, $\tr \pi(X^\beta)_{|_{V_1}}^t \pi(Y^\alpha)_{|_{V_1}}$, and $\tr \pi(Y^\beta)_{|_{V_1}}^t \pi(Y^\alpha)_{|_{V_1}}$ are all equal to zero for $\alpha\neq\beta$ in $\Delta^+$.

It remains to show that $\tr \pi(Y^\alpha)_{|_{V_1}}^t \pi(X^\alpha)_{|_{V_1}}=0$ for all $\alpha\in\Delta^+$.
This follows from the facts that $\pi(X^\alpha)$ preserves  the (orthogonal) subspaces
\begin{equation}
\Span_\RR\{v_1,\dots,v_c\}
\quad\text{and}\quad
\Span_\RR\{\mi v_1,\dots,\mi v_c\},
\end{equation}
while $\pi(Y^\alpha)$ switches them.
We conclude that \eqref{item:trace=0} holds.

We now assume that $\mathcal S$ has exactly one element, say $\mu_0$.
Then, the validity of \eqref{item:X^alphaV_1subsetV_2} follows since, for any $\alpha\in\Delta^+$, \eqref{eq:E_alphaV(mu)} yields that $\pi(X^\alpha) V_1$ and $\pi(Y^\alpha)V_1$ are included into $V(\mu_0+\alpha)\oplus V(\mu_0-\alpha)$, which is included in $V_2$ because  $\mu_0\pm\alpha\notin \mathcal S=\{\mu_0\}$.
\end{proof}

The next lemma will be very useful to check property \eqref{item:non-trivial}.

\begin{lemma}\label{lem:sl(2,C)-modulo}
Let $(\pi,V)$ be a complex representation of $\mathfrak u_\CC$, let $\mu$ be a weight of $\pi$, and let $\alpha\in\Delta$.
If the integer $\langle \alpha,\mu\rangle$ is non-zero, then $\pi(X^\alpha)_{|_{V(\mu-k\alpha)}}$ and $\pi(Y^\alpha)_{|_{V(\mu-k\alpha)}}$ are non-trivial for all $k$ between $0$ and the integer $\langle \alpha,\mu\rangle$.
\end{lemma}

\begin{proof}
It is well known that $\mathfrak a_\alpha:= \Span_\CC \{H_\alpha,X_{\alpha},X_{-\alpha}\}$ is isomorphic to $\mathfrak{sl}(2,\CC)$.
Let $v\in V(\mu)$ with $v\neq0$.

We first assume that the integer $m:=\langle \alpha,\mu\rangle$ is positive.
From the well-known representation theory of $\mathfrak{sl}(2,\CC)$ (see for instance \cite[\S{}I.9]{Kn}), since
\begin{equation}
\pi(H_\alpha) v=\mu(H_\alpha) v= \langle \mu,\alpha\rangle v=m v,
\end{equation}
we deduce that the irreducible $\mathfrak a_\alpha$-submodule of $V$ containing $v$ has dimension at least $m+1$, and moreover, $\pi(X_{-\alpha})^k v\neq0$ for all $0\leq k\leq m$.
Consequently, $\pi(X^\alpha)_{|_{V(\mu-k\alpha)}}$ and $\pi(Y^\alpha)_{|_{V(\mu-k\alpha)}}$ are non-trivial for all $0\leq k\leq m$.

We now assume that $m$ is negative.
Analogously as above, $\pi(X_{\alpha})^k v\neq0$ for all $m\leq k\leq 0$, thus $\pi(X^\alpha)_{|_{V(\mu-k\alpha)}}$ and $\pi(Y^\alpha)_{|_{V(\mu-k\alpha)}}$ are non-trivial.
\end{proof}

\subsection{Weyl chamber approach}
The following result is the first approach to use Theorem~\ref{theTheo}.
The main result of the paper, Theorem~\ref{thm:main-simple}, is proved below as a consequence of this approach.

\begin{theorem} \label{thm:Weylchamber}
Let $(\pi,V)$ be a complex representation of $\mathfrak u_\CC$.
We assume that there is a weight $\mu_0$ of $\pi$ lying in a Weyl chamber.
By setting $V_1= V(\mu_0)$ and $V_2=\bigoplus_{\mu\neq \mu_0} V(\mu)$, all the properties in Theorem~\ref{theTheo} hold.
Consequently,  the real Lie algebra $\mathfrak l(\mathfrak u,\pi)$
admits an inner product with negative Ricci curvature.
\end{theorem}

\begin{proof}
The decomposition $V=V_1\oplus V_2$ chosen is as in Proposition~\ref{prop:V_1(S)} with $\mathcal S=\{\mu_0\}$, thus it only remains to show property \eqref{item:non-trivial}.
This follows immediately from Lemma~\ref{lem:sl(2,C)-modulo} since $\langle \mu_0,\alpha\rangle \neq 0$ for all $\alpha\in\Delta^+$ because $\mu_0$ is not orthogonal to any root from being inside a Weyl chamber.
\end{proof}

We are now in position to prove the main theorem.

\begin{proof}[Proof of Theorem~\ref{thm:main-simple}]
By applying Theorem~\ref{theTheo}, Theorem~\ref{thm:Weylchamber} tell us that for every complex representation $(\pi,V)$ of $\mathfrak u_\CC$ containing a weight in a chamber Weyl, $\mathfrak l(\mathfrak u,\pi)$ admits an inner product with negative Ricci curvature.
Then, Lemma~\ref{lem:allbutfinitelymany} ensures that this property holds for all but finitely many irreducible representations of $\mathfrak u_\CC$, which completes the proof.
\end{proof}

\begin{theorem}\label{thm:main-semisimple}
	Let $\mathfrak u $ be a compact semisimple real Lie algebra.
	The real Lie algebra  $\mathfrak l(\mathfrak u,\pi)$ given by \eqref{l} admits an inner product with negative Ricci curvature for infinitely many finite-dimensional irreducible complex representations $(\pi,V)$ of $\mathfrak u_\CC$.
\end{theorem}

\begin{proof}
There are $\mathfrak u_1,\dots,\mathfrak u_m$ real compact Lie algebras such that
\begin{align*}
\mathfrak u &\simeq \bigoplus_{j=1}^m \mathfrak u^{(j)}
&
\mathfrak u_\CC&\simeq \bigoplus_{j=1}^m \mathfrak u_\CC^{(j)},
\end{align*}
and $\mathfrak u_\CC^{(j)}$ is simple for all $j$.
There is also a root system $\Delta= \bigcup_{j=1}^m \Delta^{(j)}$ of $\mathfrak u_\CC$ such that $\Delta^{(j)}$ is a root system of $\mathfrak u_\CC^{(j)}$ with associated real compact form $\mathfrak u^{(j)}$.

Every irreducible representation of $\mathfrak u_\CC$ is of the form $(\pi,V^{(1)}\otimes\dots\otimes V^{(m)})$, where $(\pi^{(j)},V^{(j)})$ is an irreducible representation of $\mathfrak u_\CC^{(j)}$ for all $j$ and
\begin{equation}
\pi(X_1,\dots,X_m) \; v^{(1)}\otimes \dots\otimes v^{(m)}=\sum_{j=1}^m v^{(1)}\otimes \dots\otimes \pi^{(j)}(X_j) v_j \otimes \dots\otimes v^{(m)}
\end{equation}
for all $X_j\in\mathfrak u_\CC^{(j)}$ and $v^{(j)}\in V^{(j)}$ for all $j$.

By Lemma~\ref{lem:allbutfinitelymany}, there are infinitely many irreducible representations $(\pi,V^{(1)}\otimes\dots\otimes V^{(m)})$ of $\mathfrak u_\CC$ such that, for every $j$, $V^{(j)}$ contains a weight  $\eta^{(j)}$ in a Weyl chamber associated to $\Delta^{(j)}$.
It follows that $\eta:=\eta^{(1)}+\dots+ \eta^{(m)}$ is in a Weyl chamber of $\Delta$.
In fact, if $\alpha\in \Delta$, then $\alpha\in\Delta^{(j)}$ for some $j$, thus $\langle \alpha,\eta\rangle = \langle \alpha,\eta^{(j)}\rangle \neq0$.
The proof follows from Theorem~\ref{thm:Weylchamber}.
\end{proof}

\begin{remark}
For $\mathfrak g$ a complex semisimple non-simple Lie algebra, Lemma~\ref{lem:allbutfinitelymany} is not longer true.
We assume for simplicity that $\mathfrak g=\mathfrak g_1\oplus \mathfrak g_2$.
We fix $V$ an irreducible complex representation of $\mathfrak g_1$ having no weight in a Weyl chamber.
Thus, for any irreducible complex representation $W$ of $\mathfrak g_2$, the irreducible complex representation $V\otimes W$ of $\mathfrak g$ does not contain any weight in a Weyl chamber, and of course, there are infinitely many such representations.
\end{remark}

\subsection{Other approaches}

The next result is our second approach to use Theorem~\ref{theTheo}.
We call it the \emph{Weyl group orbit approach} because $\mathcal S$ is the Weyl orbit of a weight of the representations satisfying certain conditions.

\begin{theorem} \label{thm:S=pesosextremales}
Let $(\pi,V)$ be a complex representation of $\mathfrak u_\CC$.
We assume there is a non-zero weight $\mu_0$ of $\pi$ such that  $\mu_0+\alpha\notin W\cdot\mu_0$ for all $\alpha\in\Delta$.
By setting
\begin{align*}
V_1&= \bigoplus_{w\in W} V(w\cdot \mu_0),& V_2&=\bigoplus_{\mu\notin W\cdot\mu_0} V(\mu),
\end{align*}
all the properties in Theorem~\ref{theTheo} hold.
Consequently,  the real Lie algebra $\mathfrak l(\mathfrak u,\pi)$
admits an inner product with negative Ricci curvature.
\end{theorem}

\begin{proof}
The decomposition $V=V_1\oplus V_2$ coincides with the one from Proposition~\ref{prop:V_1(S)} by taking $\mathcal S=W\cdot \mu_0$, the Weyl orbit of $\mu_0$.
It follows that properties \eqref{item:H^alpha-invariant}, \eqref{item:orthogonal}, \eqref{item:skewsymmetric}, and \eqref{item:trace=0} hold.

For $\alpha\in\Delta^+$ and $w\in W$, $\pi(X^\alpha)$ and $\pi(Y^\alpha)$ map $V(w\cdot\mu_0)$ to $V(w\cdot\mu_0+\alpha)\oplus V(w\cdot\mu_0-\alpha)$ by \eqref{eq:E_alphaV(mu)}.
If $V(w\cdot\mu_0\pm \alpha) \subset V_1$, then $w\cdot\mu_0\pm\alpha\in W\cdot\mu_0$, thus $\mu_0\pm w^{-1}\cdot\alpha\in W\cdot\mu_0$, which contradicts the assumption since $W(\Delta)\subset\Delta$.
Hence, $\pi(X^\alpha)$ and $\pi(Y^\alpha)$ map $V(w\cdot\mu_0)$ into $V_2$ for all $\alpha\in\Delta^+$ and $w\in W$, thus property \eqref{item:X^alphaV_1subsetV_2} holds.

It only remains to establish the validity of \eqref{item:non-trivial}.
We fix $\alpha\in\Delta^+$.
Let $w_0\in W$ satisfying that $\langle \alpha,w_0\cdot\mu_0\rangle\neq0$.
Such element exists because $\alpha\neq0$ and $\Span_\CC(W\cdot \mu_0)=\mathfrak h^*$ (since $\mu_0\neq0$).
We assume that the integer $m:=\langle \alpha,w_0\cdot\mu_0\rangle$ is positive; the negative case is analogous.

Let $v\in V(w_0\cdot\mu_0)$ with $v\neq0$.
We have that $\pi(H_\alpha) v=(w_0\cdot\mu_0)(H_\alpha)v= \langle \alpha,w_0\cdot\mu_0\rangle v = mv$.
Similarly as in the proof of Lemma~\ref{lem:sl(2,C)-modulo}, the representation theory of $\mathfrak{sl}(2,\CC)$ implies that $\pi(X_{-\alpha})^kv\neq0$ for all $0\leq k\leq m$.
In particular, $\pi(X_{-\alpha})v\neq0$, hence $\pi(X^\alpha)_{|_{V_1}}$ and $\pi(Y^\alpha)_{|_{V_1}}$ are non-trivial.
This completes the proof.
\end{proof}

The Weyl group orbit approach has been implicitly used in \cite{CRN}.
In fact, the proof of \cite[ Thm.~1.1]{CRN} decomposes any polynomial representation  $\pi$ of a classical Lie algebra as in Theorem~\ref{thm:S=pesosextremales} with $\mu_0$ the highest weight of $\pi$.

We now state the \emph{zero weight approach} which picks $\mathcal S=\{0\}$.

\begin{theorem} \label{thm:zeroweight}
Let $(\pi,V)$ be a complex representation of $\mathfrak u_\CC$.
We assume that $0$ and all the roots are weights of $\pi$.
By setting $V_1=V(0)$ and $V_2=\bigoplus_{\mu\neq0} V(\mu)$, all the properties in Theorem~\ref{theTheo} hold.
Consequently, the real Lie algebra $\mathfrak l(\mathfrak u,\pi)$
admits an inner product with negative Ricci curvature.
\end{theorem}

\begin{proof}
The decomposition $V=V_1\oplus V_2$ is as in Proposition~\ref{prop:V_1(S)} with $\mathcal S=\{0\}$, thus it only remains to check that property \eqref{item:non-trivial} holds.

Let $\alpha\in\Delta^+$.
By assumption, $\alpha$ is a weight of $\pi$, that is, $V(\alpha)\neq0$.
Since $\langle \alpha,\alpha\rangle$ is a positive integer, Lemma~\ref{lem:sl(2,C)-modulo} ensures that $\pi(X^\alpha)_{|_{V(0)}}$ and $\pi(Y^\alpha)_{|_{V(0)}}$ are non-trivial.
\end{proof}

\begin{remark} 
It is clear that the adjoint representation of $\mathfrak u_{\CC}$ satisfies the assumptions of Theorem~\ref{thm:zeroweight}. 
\end{remark}

\section{Explicit examples}\label{sec:examples}

This section contains many explicit examples of metric Lie algebras with negative Ricci operator constructed from Theorems~\ref{thm:Weylchamber}, \ref{thm:S=pesosextremales} and \ref{thm:zeroweight}.

We recall that the irreducible complex representations of $\mathfrak u_\CC$ are in correspondence with elements in $P^+(\mathfrak u)=\bigoplus_{i=1}^n \ZZ_{\geq0}\omega_i$, where $\omega_1,\dots,\omega_n$ are the fundamental weights corresponding to the simple root system $\Pi=\{\alpha_1,\dots,\alpha_n\}$.
In the sequel, when we consider a particular complex simple Lie algebra, we will use the positive root system chosen in \cite[\S{}C.1--2]{Kn}.

If $\lambda=\sum_{i=1}^n a_i\omega_i\in P^+(\mathfrak u_\CC)$ satisfies $a_i>0$ for all $i$, then $\lambda$ lies in the fundamental Weyl chamber $C^+$.
It follows immediately that $\mathfrak l(\mathfrak u,\pi_\lambda)$ admits an inner product having negative Ricci curvature from the Weyl chamber approach (Theorem~\ref{thm:Weylchamber}).
It remains to analyze those dominant integral weights lying in the faces of $C^+$, namely
$$
\left\{ \sum_{i=1}^n a_i\omega_i: a_i\in\NN\;\forall i\in I,\; a_i=0\;\forall i\notin I \right\}
$$
for each non-empty and proper subset $I$ of $\{1,\dots,n\}$.

When $\mathfrak u=\mathfrak{su}(2)$, every non-trivial dominant integral weight is in the fundamental Weyl chamber.
Consequently, $\mathfrak l(\mathfrak{su}(2),\pi)$ admits an inner product with negative Ricci curvature for all non-trivial complex irreducible representation $\pi$ of $\mathfrak u_\CC$.
Indeed, this fact was proved in \cite[\S{}3]{u2}.

We next classify, in the case when the rank of $\mathfrak u_\CC$ is two, the dominant integral weights $\lambda$ such that one of our approaches applies and therefore $\mathfrak l(\mathfrak u,\pi_\lambda)$ admits an inner product with negative Ricci operator.

\begin{proposition}\label{prop:rank2}
Let $\mathfrak u$ be a real Lie algebra such that $\mathfrak u_\CC$ is simple of rank two.
Then, $\mathfrak l(\mathfrak{u}, \pi_\lambda)$ admits an inner product with negative Ricci curvature for every $\lambda\in P^+(\mathfrak u_\CC)$ except possibly $\{0,\omega_1,\omega_2\}$ for types $\textup{A}_2$ and $\textup{B}_2=\textup{C}_2$, and $\{0,\omega_1\}$ for type $\textup{G}_2$.
\end{proposition}

\begin{proof}

For short, we say that a complex representation $\pi$ of $\mathfrak u_\CC$ has Ricci negative curvature if $\mathfrak l(\mathfrak{u}, \pi_\lambda)$ admits an inner product with negative Ricci operator.

Every dominant integral weight is of the form $a\omega_1+b\omega_2$ for some non-negative integers $a,b$.
From the above discussion, the cases $ab\geq1$ follow from the Weyl chamber approach.
Hence, it only remains to check the cases $a\omega_1$ and $a\omega_2$ for $a\in\NN$.

We first assume that $\mathfrak u_\CC$ is of type $\textup{G}_2$.
Let $\alpha_1$ and $\alpha_2$ denote the simple roots, thus $\Delta^+=\{ \alpha_1,\alpha_2,\alpha_1+\alpha_2, 2\alpha_1+\alpha_2, 3\alpha_1+\alpha_2, 3\alpha_1+2\alpha_2 \}$ and the fundamental weights are $\omega_1=2\alpha_1+\alpha_2$ and $\omega_2=3\alpha_1+2\alpha_2$.

From the Weyl chamber approach, $\pi_{a\omega_1}$ (resp.\ $\pi_{a\omega_2}$) has negative Ricci operator if $a\geq3$ ($a\geq2$) since $\lambda-\alpha_1\in P_\lambda(\mathfrak u_\CC)\cap C^+$ ($\lambda-\alpha_1-\alpha_2\in P_\lambda(\mathfrak u_\CC)\cap C^+$).
The case $\lambda=\omega_2$ follows from the zero weight approach (Theorem~\ref{thm:zeroweight}) since $P_{\lambda}(\mathfrak u_\CC)= \Delta\cup \{0\}$.
The case $\lambda=2\omega_1$ follows from the Weyl orbit approach (Theorem~\ref{thm:S=pesosextremales}) since $\lambda+\alpha\notin W\cdot\lambda$ for all $\alpha\in\Delta$.

The rest of the cases follows from the more general result in Lemma~\ref{lem:fundweight} below.
\end{proof}

\begin{lemma}\label{lem:fundweight}
Let $\mathfrak u$ be a compact real simple Lie algebra such that $\mathfrak u_\CC$ is of classical type and rank $n$.
Then $\mathfrak l(\mathfrak u,\pi_{a\omega_p})$ admits an inner product with negative Ricci curvature for every $a\geq2$ and $1\leq p\leq n$.
\end{lemma}

\begin{proof}
Let $\mathfrak u_\CC$ be the simple complex Lie algebra of type $\textup{C}_n$ for some $n\geq2$.
There is a $\CC$-basis $\{\varepsilon_1,\dots, \varepsilon_{n}\}$ of $\mathfrak h^*$ such that $\Delta^+= \{\varepsilon_i\pm\varepsilon_j : 1\leq i<j\leq n\} \cup \{\varepsilon_{i}:1\leq i\leq n\}$, $\omega_p=\varepsilon_1+\dots+\varepsilon_p$ for any $1\leq p\leq n$, $W \simeq \mathbb S^n\times \{\pm1\}^n$ and the element $(\sigma,\{t_i\}_{i=1}^n)$ acts by $\sum_{i=1}^n a_i\varepsilon_i \mapsto \sum_{i=1}^n t_i a_{\sigma(i)}\varepsilon_i $.
Thus $W\cdot a\omega_p=\{\pm a\omega_i:1\leq i\leq n\}$.
It follows immediately that $a\omega_p+\alpha\notin W\cdot a\omega_p$ for all $\alpha\in\Delta$.
Thus, the Lemma follows by Theorem~\ref{thm:S=pesosextremales}.

The rest of the cases are very similar.
\end{proof}

\begin{remark}
One can easily check that the non-trivial exceptions in Proposition~\ref{prop:rank2} (i.e.\ $\pi_{\omega_1},\pi_{\omega_2}$ for types $\textup{A}_2$ and $\textup{B}_2$, and $\pi_{\omega_1}$ for type $\textup{G}_2$) do not follow from any of the three approaches.
In case it is possible, it would be interesting to find in any of the corresponding Lie algebras an inner product with negative Ricci curvature.
We note that beyond the solvable case there is no necessary condition in the literature for a Lie group to admit a metric with negative Ricci curvature. This makes the problem to prove that a Lie algebra does not admit such a metric, a challenging one.
\end{remark}

\section{General nilradical}\label{sec:generalnilradical}

Finally, we can use the same idea as in \cite[Thm.~5.4]{CRN} to get examples with a non-abelian nilradical. Explicitly, we will consider the more general setting of a Lie algebra $\ggo= (\RR Z \oplus \ug) \ltimes \ngo$  where $\ug$ is a compact semisimple Lie algebra acting on a nilpotent Lie algebra $\ngo$ by derivations and as always, $Z$ commutes with $\ug$.
Note that in order to get a Lie algebra, $\ad Z$ must be a derivation of $\ngo$ and therefore could not act as a multiple of the identity unless $\ngo$ is abelian. 

First we will prove a more general version of Theorem \ref{theTheo}.

\begin{theorem}\label{Z:positive}
Let $\ug$ be a compact semisimple Lie algebra, and let $(\pi,V)$ be a finite dimensional real representation of $\mathfrak u$.
We assume $V$ decomposes in $\mathfrak u$-submodules as $V= W_1 \oplus \dots \oplus W_k$ such that, for some index $i$, there exists a decomposition $W_i = V'_1 \oplus V'_2$ and a real inner product on $W_i$ satisfying properties (1)--(6) from Theorem~\ref{theTheo}.
Then, for all positive real numbers $c_1,\dots,c_k$, the Lie algebra $(\RR Z \oplus \ug) \ltimes V$ determined by ${\ad Z}_{|_\ug}=0$ and ${\ad Z}_{|_{W_i}}=c_i\Id_{W_i}$ admits a Ricci negative inner product.
\end{theorem}

\begin{proof}
Assume $W_i= V'_1 \oplus V'_2$, $\langle\cdot,\cdot \rangle_i$ satisfies (1)--(6) from Theorem~\ref{theTheo}. It is not hard to check that all these hypotheses remain valid for $V=V_1 \oplus V_2$ and $\langle \cdot,\cdot \rangle$ where $V_1=V_1'$, $V_2 = V_2' \oplus \sum_{j\ne i} W_j$ and $\langle \cdot,\cdot \rangle$ is obtained by extending $\langle \cdot,\cdot \rangle_i$ to $V$ in such a way that $V= W_1 \oplus \dots \oplus W_k$ is an orthogonal decomposition and such that every element of $\mathfrak u$ acts as a skew-symmetric operator of $\sum_{j\ne i} W_j$ (see (\ref{inner_product})).
Then, Theorem~\ref{theTheo} proves the theorem for the case $c_1=\dots=c_k=1$.
Moreover, the case $c_1=\dots=c_k$ also follows from Theorem~\ref{theTheo} since positive rescaling at the element $Z$ in the definition \eqref{l} of $\mathfrak l(\mathfrak u, \pi)$ returns isomorphic Lie algebras.

For the general case, we can follow the same proof as in \cite[Thm.~3.3]{CRN} with minor changes coming from the fact that we have a different mean curvature vector on $\widetilde{\mathfrak l}$ even though it still is a multiple of $Z$.
We next sketch this proof by following the notation in \cite{CRN}.

Let $\widetilde{\mathfrak{l}}$
denote the Lie algebra $(\RR Z \oplus \ug) \ltimes V$ determined by ${\ad Z}_{|_\ug}=0$ and ${\ad Z}_{|_{W_i}}=c_i\Id_{W_i}$.
We first note that the same degeneration given in \cite[Lem.~3.1]{CRN} shows that $\widetilde{\mathfrak l}$ degenerates into the solvable Lie algebra $\widetilde{\mathfrak l}_\infty=(\RR Z \oplus \ug \oplus V, \mu)$ which only differs from  ${\mathfrak l}_\infty$ by the action of $Z$.
The next step is to show that $\widetilde{\mathfrak l}_\infty$ admits an inner product with negative Ricci operator.
Note that the mean curvature vector of $\widetilde{\mathfrak l}_\infty$ is
$$
H = c \, Z,\quad \text{ where } \, c=\tr (\ad_\mu Z) = \sum_{i=1}^k c_i \dim W_i.
$$
We proceed by noting that
$- c \ad Z$ is negative definite on $V$ and diagonalizes in any basis formed as a union of basis of $W_j$, so one can find an appropriated $\rho$.  Finally note that because of this diagonalization property, we can perturb the inner product on $V_1$ and still get a block diagonal Ricci operator as in \cite{CRN}.
\end{proof}

Let $\mathfrak g$ be a non-solvable real Lie algebra.
From its Levi decomposition, we know there is a semisimple Lie subalgebra $\mathfrak u$ of $\mathfrak g$ (called the Levi subalgebra of $\mathfrak g$) such that $\mathfrak g\simeq \mathfrak u\ltimes \mathfrak s$, where $\mathfrak s$ is the radical of $\mathfrak g$.
In the next theorem, we will assume that $\mathfrak u$ is compact and the nilradical $\mathfrak n$ of $\mathfrak s$ has codimension one.
It is not hard to see that one can always get a complement $\mathfrak{a}$ of $\mathfrak n$ in $\mathfrak s$ such that $[\mathfrak{u}, \mathfrak{a}]=0$ (see \cite{CRN} after Theorem 5.2).
Thus, for a non-zero element $Z \in \mathfrak a$, we have that $\mathfrak g\simeq (\RR Z\oplus \mathfrak u)\ltimes \mathfrak n$, with $[Z,\mathfrak u]=0$.

\begin{theorem}\label{thm:nilpotent}
Let $\ggo= (\RR Z \oplus \ug) \ltimes \ngo$ be a real Lie algebra where  $\mathfrak u$ is a compact semisimple Lie algebra, $\ngo$ is a nilpotent Lie algebra and
$Z$ is a non-zero element commuting with $\mathfrak u$.
Let $\ngo=\ngo_1\oplus \dots \oplus \ngo_k$ be the decomposition of $\ngo$ as an $\mathfrak u$-module in irreducible components.
If any of the $\ngo_j$ admits a decomposition as in Theorem~\ref{theTheo} and $\ad Z_{|_\ngo}$ acts as a positive multiple of the identity in every $\ngo_i$,
then $\ggo$ admits an inner product with negative Ricci curvature.
\end{theorem}

\begin{proof}
We know that $\ad Z_{|_{\ngo_i}}=c_i \Id_{\ngo_i}$ for some $c_i>0$.
We set $\mathfrak a=\RR Z$.

It is shown in \cite[Theorem 5.4 or (31)]{CRN}  that $\ggo=(\RR Z \oplus \ug \oplus \ngo, [,])$ degenerates into a Lie algebra $\lgo = (\RR Z \oplus \ug \oplus \ngo, [,]_0)$ where the action of $\RR Z \oplus \ug$ on $\ngo$ is the same as in $\ggo$ but $\ngo$ is an abelian ideal of $\lgo$. In fact,  for each $t>0$ let us consider $\psi_t \in \glg(\ggo)$ such that
$$
{\psi_t}_{|_{\ag \oplus \ug}}=\Id, \qquad {\psi_t}_{|_\ngo}=t \,\Id.
$$
Using that $\ngo$ is an ideal of $\ggo$, it is easy to check that $[,]_t := \psi_t \cdot [,]$ is given by
\begin{equation}\label{degabel}
\begin{aligned}{}
[X_1,X_2]_t&=[X_1,X_2]
	&&\text{ for } X_1,X_2 \in \ag\oplus \ug,
\\
[X_1,X_2]_t&=\tfrac{1}{t}[X_1,X_2]
	&&\text{ for } X_1,X_2 \in \ngo,
\\
[X_1,X_2]_t&=[X_1,X_2]
	&&\text{ for } X_1 \in \ag\oplus \ug,\; X_2 \in \ngo.
\end{aligned}
\end{equation}
Hence, $\displaystyle{\lim_{t \to \infty}} [,]_t =: [,]_0$ is well defined and is given by
\begin{equation}
\begin{aligned}{}
[X_1,X_2]_0 &=[X_1,X_2]  &&\text{for } X_1 \in  \mathfrak a\oplus \ug, X_2 \in \ggo,
\\
[X_1,X_2]_0&=0 &&\text{for } X_i  \in \ngo.
\end{aligned}
\end{equation}

 Note that since $[\mathfrak n,\mathfrak n]_0=0$, the Lie algebra $(\RR Z \oplus \ug \oplus \ngo, [,]_0)$ is an $\mathfrak l(\mathfrak u,\pi)$ as in \eqref{l}.
Now, by Theorem~\ref{Z:positive}, it admits an inner product with negative Ricci curvature and therefore so does $\ggo$.
\end{proof}

We note that when $\ad Z_{|_{\ngo}}$ is a positive definite operator with respect to some inner product on $\ngo$ or equivalently it is diagonalizable with positive eigenvalues, if it preserves an irreducible
submodule, then it necessarily acts as a multiple of the identity.
In fact, in this case
there always exist a decomposition of $\ngo$ in irreducible submodules such that
$\ad Z$ acts as a positive multiple of the identity in every term. Hence, with the same idea of the above proof we obtain the following theorem.

\begin{theorem}
Let $\ggo= (\RR Z \oplus \ug) \ltimes \ngo$ be a real Lie algebra where $\mathfrak u$ is a compact semisimple Lie algebra, $\ngo$ is a nilpotent Lie algebra and
$Z$ is a non-zero element commuting with $\mathfrak u$.
Assume that $\ad Z_{|_\ngo}$ is diagonalizable with positive eigenvalues and $\ngo=\ngo^{(1)}\oplus \dots \oplus \ngo^{(k)}$ is the corresponding decomposition in eigenspaces.
If any of the irreducible components of $\ngo^{(i)}$ as an $\mathfrak u$-module admits a decomposition as in Theorem~\ref{theTheo}, then $\ggo$ admits an inner product with negative Ricci curvature.
\end{theorem}

\begin{proof}
Since $\ngo=\ngo^{(1)}\oplus \dots \oplus \ngo^{(k)}$ is the decomposition in eigenspaces of $\ad Z_{|_\ngo}$ and $[Z,\mathfrak u]=0$,  $\mathfrak u$ preserves each $\ngo^{(i)}$ and one therefore gets
$$
\ngo=
\ngo_{1}^{(1)}\oplus  \dots \ngo_{j_1}^{(1)}
\oplus \dots\oplus
\ngo_{1}^{(k)}\oplus\dots \oplus \ngo_{j_k}^{(k)},
$$
a decomposition of $\ngo$ in irreducible submodules and $\ad Z_{|{\ngo_{j}^{(i)}}}=c_i \Id_{\ngo_{j}^{(i)}}$ for some $c_i>0$.

The rest of the proof follows as in Theorem~\ref{thm:nilpotent}.
\end{proof}

\end{document}